\documentclass[10pt]{amsart} 
\usepackage{tikz-cd}
  
\usepackage{amsmath,amsthm,amssymb,mathrsfs,xypic}
 
\renewcommand{\geq}{\geqslant}
\renewcommand{\leq}{\leqslant}
\newcommand{\Osh}{{\mathcal O}}                        

\newcommand{\K}{\mathrm{K}}      
                      
\renewcommand{\emptyset}{\varnothing}
\newcommand{\KK}{\mathbf{K}}
\newcommand{\FF}{\mathbf{F}}
\newcommand{\PP}{\mathbb{P}} 
\newcommand{\QQ}{\mathbb{Q}} 
\newcommand{\ZZ}{\mathbb{Z}} 

\newtheorem{theorem}{Theorem}[section]
\newtheorem{lemma}[theorem]{Lemma}

\newtheorem{conjecture}[theorem]{Conjecture}
\theoremstyle{definition}

\newtheorem{remark}[theorem]{Remark}
\newtheorem{example}[theorem]{Example}

\numberwithin{equation}{section}

\begin{document}

\title[On Vojta's harder implication for admissible pairs]{On Vojta's harder implication, height inequalities for admissible pairs, points of bounded degree and Deligne-Mumford stacks}

\author{Nathan Grieve}

\address{
Department of Mathematics, National Taiwan University, Astronomy and Mathematics Building 5F, No. 1, Sec. 4, Roosevelt Rd., Taipei 10617, Taiwan (R.O.C.); School of Mathematics and Statistics, 4302 Herzberg Laboratories, Carleton University, 1125 Colonel By Drive, Ottawa, ON, K1S 5B6, Canada; 
D\'{e}partement de math\'{e}matiques, Universit\'{e} du Qu\'{e}bec \`a Montr\'{e}al, Local PK-5151, 201 Avenue du Pr\'{e}sident-Kennedy, Montr\'{e}al, QC, H2X 3Y7, Canada
}
\email{nathan.m.grieve@gmail.com}%

\begin{abstract}
We expand on the concept of \emph{admissible pairs} from \cite{Levin:GCD} and explore its relation with the main Diophantine arithmetic inequalities, with discriminant term and for points of bounded degree, that have been predicted by Vojta \cite{Vojta:1998}.  In this context, among other new results, we prove a \emph{harder implication} which is in the spirit of Vojta's approach to the abc Conjecture (from \cite{Vojta:1998}).  As our main result, and application of our viewpoint here, we deduce for the case of certain general type nonsingular Deligne-Mumford stacks, with projective course moduli space, a form of the Bombieri-Lang Conjecture for $(D_0,S)$-integral points of bounded degree.   A key input for this is a slicing theorem, for Deligne-Mumford stacks, that was obtained by Abramovich and V\'{a}rilly-Alvarado, \cite{Abramovich:VarillyAlvarado:Pera:2017}, and building on earlier work of Kresch and Vistoli \cite{Kresch:Vistoli:2004}.  Another important ingredient is an inequality of Silverman, from \cite{Silverman:1984}, which bounds the discriminant of points in projective space in terms of their heights.  As an illustration of our results, we discuss them within the context of the interesting work of Abramovich and Harris \cite{Abramovich:Harris:1991} and others.
\end{abstract}

\thanks{
The author thanks the Natural Sciences and Engineering Research Council of Canada for their support through his grants, DGECR-2021-00218 and RGPIN-2021-03821, and also the National Science and Technology Council (Taiwan) for their support through his grants 115-2115-M-002-003-MY3 and 115-2811-M-002-043.\\
ORCID: https://orcid.org/0000-0003-3166-0039
\\
\today.  
}

\maketitle

\section{Introduction}

Our purpose here, is to initiate a general study of $(D_0,S)$-integral points of bounded degree in a nonsingular and geometrically irreducible projective variety $X$ defined over a base number field $\KK$.  Here, $D_0$ is a normal crossings divisor on $X$ and $S \subset M_{\KK}$ is a finite subset of $\KK$'s places that contains all infinite places.  Our main results are Theorems \ref{Vojta:logical:implications} and \ref{lang:conjecture:bounded:degree}.  They are derived as consequences of conjectured Diophantine arithmetic height inequalities that are in the spirit of those that have been proposed by Vojta (\cite[Conjecture 2.1]{Vojta:1998} and \cite[Conjecture 1.13]{Levin:GCD}).  Our approach here also yields new Diophantine arithmetic applications for Deligne-Mumford stacks (see Theorems \ref{admissible:pair:DM:stack:conj} and \ref{DM:Bombieri-Lang}).

To place matters into perspective, it is helpful to recall the more recent results about points of bounded degree on curves.  The starting point is Faltings' Diophantine approximation theorem for Abelian varieties \cite{Faltings:1991}.  The theorem asserts that if a subvariety of an Abelian variety does not contain any translate of a positive dimensional Abelian subvariety, then it must contain at most finitely many $\KK$-rational points.  In light of this, given a genus $g$ curve $C$ over $\KK$ together with a finite extension $\FF / \KK$ and a positive integer $g \geq 1$, following the viewpoint of Abramovich and Harris, \cite{Abramovich:Harris:1991}, let $\Gamma_{C,d}(\FF)$ denote the set of $C$'s algebraic points that have degree at most equal to $d$ over $\FF$.   The work \cite{Abramovich:Harris:1991} raises the question of characterizing when $\Gamma_{C,d}(\FF)$ is infinite.  As emphasized in \cite{Abramovich:Harris:1991} the main subtle point is the extent to which the \emph{Brill-Noether loci}
$$
W_d(C) := \{L \in \operatorname{Pic}^d(C) : h^0(C,L) \geq 1 \} \subseteq \operatorname{Pic}^d(C)
$$
contain translates of positive dimensional subvarieties of $\operatorname{Pic}^d(C)$.

Based on this principle, the main more recent results about the sets $\Gamma_{C,d}(\FF)$ are summarized in the following way.
(i) If $d =2$, then $\#\Gamma_{C,2}(\FF) = \infty$ if and only if $C$ is hyperelliptic or bielliptic. (See \cite{Harris:Silverman:1991}.) (ii)
If $1 \leq d \leq 4$ and if $g \not = 7$ if $d = 4$, then $\# \Gamma_{C,d}(\FF) = \infty$ if and only if $C$ admits a map of degree at most $d$ to $\PP^1$ or an elliptic curve and defined over $\overline{\KK}$.  (See \cite{Abramovich:Harris:1991}.)
(iii)  For all $d \geq 4$, there exist curves $C$ with $\#\Gamma_{C,d}(\KK) = \infty$ but which do not admit a map of degree $d$ or less to $\PP^1$ or an elliptic curve.  (See \cite{Debarre:Fahlaoui:1993}.)
(iv) 
If $\# \Gamma_{C,d}(\KK) = \infty$, then $C$ admits a map of degree at most $2d$ to $\PP^1$ and defined over $\KK$.  (See \cite{Frey:1994}.)
(v) Denoting the \emph{arithmetic degree of irrationality} by
$$\operatorname{a.irr}_{\KK}(C) := \min \{d : \# \Gamma_{C,d}(\KK) = \infty\}$$ 
and the \emph{gonality} by
$$\operatorname{gon}_{\KK}(C) := \min \{d : \text{$C$ admits a dominant rational degree $d$ map $C \rightarrow \PP^1_{\KK}$} \} $$ 
then 
$$
\frac{\operatorname{gon}_{\KK}(C)}{2} \leq \operatorname{a.irr}_{\KK}(C) \leq \operatorname{gon}_{\KK}(C) \text{.}
$$ (See \cite{Smith:Vogt:2022}.)

In terms of integral points of bounded degree, the work of Levin \cite{Levin:2016} gives necessary and sufficient conditions for a nonsingular projective completion of a non-singular affine curve to possess infinitely many integral points of bounded degree.  Similar to the techniques used to prove the above mentioned results about the sets $\Gamma_{C,d}(\FF)$, the approach of Levin is based on the study of the $d$-fold symmetric product $C^{(d)}$ combined with a suitable application of Faltings' result from \cite{Faltings:1991}.

Our methods here, to study $(D_0,S)$-integral points of bounded degree in nonsingular projective varieties,  are based on suitable applications of certain variants of
 \cite[Conjecture 2.1]{Vojta:1998} and  \cite[Conjecture 1.13]{Levin:GCD}.  We formulate these new conjectural statements here.  Our framework is sufficiently more general in that it allows for a formulation that is within the context of Deligne-Mumford stacks.  This builds on   \cite[Section 3]{Abramovich:VarillyAlvarado:Pera:2017} and \cite[Theorem 1]{Kresch:Vistoli:2004}; a key point is stated as Theorem \ref{DM:slicing:lemma}.   

Let us recall a statement of the celebrated Vojta's Main Conjecture with discriminant term and for points of bounded degree.
\begin{conjecture}[{Compare with \cite[Conjecture 2.1]{Vojta:1998}}, {\cite[Conjecture 14.4.14]{Bombieri:Gubler}}]\label{Vojta:Main:Conj:Disc:Bounded:Degree}
Let $D_1$ be a normal crossings divisor on a nonsingular complete variety $X$.  Let $L$ be a big line bundle on $X$.  Let $d \in \ZZ_{>0}$ and let $\epsilon > 0$.  Then there exists a proper Zariski closed subset 
$Z \subsetneq X$
which is such that
\begin{equation}\label{Vojta:counting:inequality}
n_S(D_1,x) + \mathrm{d}_{\KK}(x) \geq h_{\K_X + D_1}(x) - \epsilon h_L(x) - \mathrm{O}(1)
\end{equation}
for all 
$x \in X(\overline{\KK}) \setminus Z$
which have the property that
$[\KK(x):\KK] \leq d \text{.}$
\end{conjecture}

In the inequality \eqref{Vojta:counting:inequality}, $n_S(D_1,x)$ is the \emph{counting function}.  Recall, its relation to the \emph{proximity function} $m_S(D_1,x)$ via the \emph{height function} $h_{\Osh_X(D_1)}(x)$
\begin{equation}\label{eqn:2}
h_{\Osh_X(D_1)}(x) = m_S(D_1,x) + n_S(D_1,x) + \mathrm{O}(1) \text{.}
\end{equation}

Next, we discuss Levin's concepts of \emph{admissible pair} $(X,U)$ and $(D_0,S)$-integral points \cite[Section 1.2]{Levin:GCD}.  In particular, $D_0$ is a normal crossings divisor on a nonsingular complete variety $X$ and $U := X \setminus D_0 \text{.}$  
The data $(X,U)$ is the \emph{admissible pair} that is associated to the complete nonsingular normal crossings pair $(X,D_0)$.  

Within this context, as in \cite[Section 1.2]{Levin:GCD}, we say that a subset 
\begin{equation}\label{eqn:3}
R \subset X(\overline{\KK}) \setminus \operatorname{Supp} D_0
\end{equation}
constitutes a collection of \emph{$(D_0,S)$-integral points} if 
\begin{equation}\label{eqn:4}
m_S(D_0,x) = h_{\Osh_X(D_0)}(x) + \mathrm{O}(1) \text{ for all $x \in R$.}
\end{equation}

In the spirit of Conjecture \ref{Vojta:Main:Conj:Disc:Bounded:Degree}, we propose the following more general formulation of \cite[Conjecture 1.13]{Levin:GCD}.
\begin{conjecture}[Compare with {\cite[Conjecture 1.13]{Levin:GCD}}]\label{Admissible:pairs:Vojta:Main:Conj:Disc:Bounded:Degree}
Let $D_0$ be a normal crossings divisor on a nonsingular complete variety $X$.  Let $D_1$ be an effective divisor on $X$ which has the property that $D_1+D_0$ is a normal crossings divisor.  Let $L$ be a big line bundle on $X$.  Let $d \in \ZZ_{>0}$ and $\epsilon > 0$.  Then there exists a proper Zariski closed subset $Z \subsetneq X$ 
which is such that for each set of $(D_0,S)$-integral points
\begin{equation}\label{eqn:5}
R \subseteq X(\overline{\KK}) \setminus \operatorname{Supp} D_0
\end{equation}
the inequality
\begin{equation}\label{eqn:5a}
n_S(D_1,x) + \mathrm{d}_{\KK}(x) \geq h_{\K_X + D_0+D_1}(x) - \epsilon h_L(x) - \mathrm{O}(1)
\end{equation}
is valid for all $x \in R \setminus Z$ 
with the property that $[\KK(x):\KK] \leq d \text{.}$
\end{conjecture}

By analogy with \cite[Conjecture 2.3]{Vojta:1998}, we may also consider a \emph{truncated form} of Conjecture \ref{Admissible:pairs:Vojta:Main:Conj:Disc:Bounded:Degree}.

\begin{conjecture}\label{truncated:Admissible:pairs:Vojta:Main:Conj:Disc:Bounded:Degree} Given the hypothesis of Conjecture \ref{Admissible:pairs:Vojta:Main:Conj:Disc:Bounded:Degree}, the conclusion of Conjecture \ref{Admissible:pairs:Vojta:Main:Conj:Disc:Bounded:Degree} holds true with the inequality \eqref{eqn:5a} replaced by the truncated inequality
\begin{equation}\label{eqn:6}
n_S^{(1)}(D_1,x) + \mathrm{d}_{\KK}(x) \geq h_{\K_X + D_0+D_1}(x) - \epsilon h_L(x) - \mathrm{O}(1) \text{.}
\end{equation}
\end{conjecture} 

In the spirit of \cite[Theorem 3.1]{Vojta:1998}, our first result establishes the logical equivalence of Conjectures \ref{Vojta:Main:Conj:Disc:Bounded:Degree}, \ref{Admissible:pairs:Vojta:Main:Conj:Disc:Bounded:Degree} and \ref{truncated:Admissible:pairs:Vojta:Main:Conj:Disc:Bounded:Degree}.

\begin{theorem}\label{Vojta:logical:implications} 
There exists the following sequence of logical implications
$$
\text{
 Conjecture \ref{Vojta:Main:Conj:Disc:Bounded:Degree} $\iff$ 
Conjecture \ref{Admissible:pairs:Vojta:Main:Conj:Disc:Bounded:Degree} $\iff$ Conjecture \ref{truncated:Admissible:pairs:Vojta:Main:Conj:Disc:Bounded:Degree}.
}
$$
\end{theorem}

We prove Theorem \ref{Vojta:logical:implications} in Section \ref{Section:proof:harder:implication}.  In a complementary direction, we now discuss topics that surround Theorem \ref{Vojta:logical:implications} within the context of Diophantine arithmetic features of Deligne-Mumford stacks.  We refer to the text \cite{Olsson:Stacks:2016} for the most basic theory of Deligne-Mumford stacks.  

Our starting point, here, is a covering theorem that can be extracted from the work of Abramovich and V\'{a}rilly-Alvarado \cite[Proposition 2.3 and Proof of Proposition 3.2]{Abramovich:VarillyAlvarado:Pera:2017}. 

\begin{theorem}[Compare with {\cite[Proposition 2.3 and Proof of Proposition 3.2]{Abramovich:VarillyAlvarado:Pera:2017}} and {\cite[Theorem 1]{Kresch:Vistoli:2004}}]\label{DM:slicing:lemma}
Suppose that $\mathcal{X}$ is a nonsingular Deligne-Mumford stack with projective coarse moduli scheme.  The following assertions hold true.
\begin{enumerate}
\item[(i)]{$\mathcal{X}$ is a quotient stack and admits a finite surjective morphism
\begin{equation}\label{DM:Covering:Slicing:Lemma}
\pi \colon Y \rightarrow \mathcal{X}
\end{equation}
from a nonsingular projective $\KK$-variety $Y$. 
}
\item[(ii)]{If 
$\mathcal{D} \subset \mathcal{X}$ 
is a normal crossings divisor, then the morphism \eqref{DM:Covering:Slicing:Lemma} can be chosen in such a way that
$
D' := \pi^* \mathcal{D} \subset Y
$
is a normal crossings divisor and such that the ramification divisor meets each stratum of $D'$ properly.
}
\item[(iii)]{If 
$x \in \mathcal{X}(\overline{\KK})$
has the property that 
$[\KK(x) : \KK] \leq d \text{,}$ then $x$ is the image of some 
$y \in Y(\overline{\KK})$ 
which has the property that 
$[\KK(y) : \KK] \leq d \cdot \operatorname{deg}(\pi) \text{.}$
}
\end{enumerate}
\end{theorem}

\begin{proof}
See \cite[Proposition 2.3]{Abramovich:VarillyAlvarado:Pera:2017} and \cite[Proof of Proposition 3.2]{Abramovich:VarillyAlvarado:Pera:2017}.
\end{proof}

Theorem \ref{DM:slicing:lemma} allows for Diophantine arithmetic questions on $\mathcal{X}$ to be formulated in terms of the cover \eqref{DM:Covering:Slicing:Lemma}.  For example, Diophantine arithmetic questions for points $x \in \mathcal{X}(\KK)$ translate into Diophantine arithmetic questions for those points of the cover $y \in Y(\overline{\KK})$ that have degree at most $\operatorname{deg}(\pi)$ over $\KK$.

This principle is expanded upon in Theorem \ref{admissible:pair:DM:stack:conj} below.  It can be compared with \cite[Proposition 3.2]{Abramovich:VarillyAlvarado:Pera:2017} and is our formulation of the equivalent Conjectures \ref{Admissible:pairs:Vojta:Main:Conj:Disc:Bounded:Degree} and \ref{truncated:Admissible:pairs:Vojta:Main:Conj:Disc:Bounded:Degree} for Deligne-Mumford stacks.  

\begin{theorem}\label{admissible:pair:DM:stack:conj}
Suppose that $\mathcal{X}$ is a nonsingular Deligne-Mumford stack with projective coarse moduli scheme $X$.  Let $\mathcal{D}_0$ be a normal crossings divisor on $\mathcal{X}$ and let $\mathcal{D}_1$ be an effective divisor on $\mathcal{X}$ which has the property that $\mathcal{D}_0 + \mathcal{D}_1$ is a normal crossings divisor on $\mathcal{X}$.  Fix a finite surjective morphism 
$\pi \colon Y \rightarrow \mathcal{X}$ 
from a nonsingular projective $\KK$-variety $Y$ which has the property that
\begin{enumerate}
\item[(i)]{ 
$
D' := D_0' + D_1' = \pi^* \mathcal{D}_0 + \pi^* \mathcal{D}_1
$
is a normal crossings divisor on $Y$; and
}
\item[(ii)]{the ramification divisor 
$\operatorname{Ram}(\pi)$
meets each stratum of $D'$ properly.   
}
\end{enumerate}
Let $\mathcal{L}$ be a big line bundle on $\mathcal{X}$ and let $L' = \pi^*\mathcal{L} \text{.}$  Assume that the conclusion of Conjecture \ref{truncated:Admissible:pairs:Vojta:Main:Conj:Disc:Bounded:Degree} holds for $Y$ with respect to $D_1'$, $L'$ and $(D_0',S)$-integral points on $Y$.  Then the conclusion of Conjecture \ref{truncated:Admissible:pairs:Vojta:Main:Conj:Disc:Bounded:Degree} holds for $\mathcal{X}$ with respect to $\mathcal{D}_1$, $\mathcal{L}$ and $(\mathcal{D}_0,S)$-integral points.

In more precise terms, let $d \in \ZZ_{>0}$ and $\epsilon > 0$.    Then there exists a proper Zariski closed subset 
$\mathcal{Z} \subsetneq \mathcal{X}$ 
such that if $Z' = \pi^{-1} (\mathcal{Z})$ 
then for each set of $(D_0,S)$-integral points
$R' \subseteq Y(\overline{\KK}) \setminus \operatorname{Supp} D_0'$
the inequality
$$
n_S^{(1)}(D_1,y') + \mathrm{d}_{\KK}(y') \geq h_{\K_Y + D_0'+D_1'}(y') - \epsilon h_L(y') - \mathrm{O}(1)
$$
is valid for all $y' \in R' \setminus Z'$ 
with the property that $[\KK(y'):\KK] \leq d \cdot \operatorname{deg}(\pi)$ 
and $\pi(y') \in X \setminus Z \text{.}$
\end{theorem}

\begin{proof}
The existence of such a suitable finite morphism $\pi \colon Y \rightarrow \mathcal{X}$ 
follows from Theorem \ref{DM:slicing:lemma}.  Working on $Y$, the conclusion of Theorem \ref{admissible:pair:DM:stack:conj} then follows from the conclusion of Conjecture \ref{truncated:Admissible:pairs:Vojta:Main:Conj:Disc:Bounded:Degree} applied to $Y$ with respect to $D_1'$, $L'$ and $(D_0',S)$-integral points.
\end{proof}

It is helpful to note that the approach of \cite[Section 3]{Abramovich:VarillyAlvarado:Pera:2017} is to formulate Diophantine arithmetic inequalities directly on the underlying stack.  In doing so, some additional, more technical, objects arise.  These include a stack theoretic discriminant term and also the absence of a Northcott property for big line bundles.
Our approach, is to work on the cover \eqref{DM:Covering:Slicing:Lemma} directly.  Among other features, this avoids having to deal with these technicalities explicitly.  

Further, as is clear upon examining the proof \cite[Proposition 3.2]{Abramovich:VarillyAlvarado:Pera:2017}, there is no loss of Diophantine arithmetic complexity by working on a such a cover \eqref{DM:Covering:Slicing:Lemma}.  Let us also mention that our approach can be used to formulate, given suitable Schmidt Subspace Theorem inequalities, with discriminant term and for points of bounded degree, on the cover \eqref{DM:Covering:Slicing:Lemma}, analogous Subspace Theorem type inequalities on the underlying Deligne-Mumford stack.    

Finally, it is important to mention the recent work \cite{Ellenberg:et:al:2023} which develops a theory of heights for Deligne-Mumford stacks which differs both from the approach that we employ here and also from the approach developed in \cite{Abramovich:VarillyAlvarado:Pera:2017}.  While, for the purpose of the present article, we refrain from formulating such results explicitly, we do note that two recent articles which pertain to this latter question include \cite{Grieve:points:bounded:degree} and \cite{Le:Giang:2023}.  They build on earlier works including \cite{Schlickewei:2003} and \cite{Levin:2014}.  Another recent related work is \cite{Grieve:qualitative:subspace} which, among other results, develops a theory of twisted Weil functions and, as an application, establishes a parametric form of the Subspace Theorem, with linear scattering, for linear systems. 

Our application of Theorems \ref{DM:slicing:lemma} and \ref{admissible:pair:DM:stack:conj}, applies to particular cases in which the canonical class $\K_{\mathcal{X}}$ is big.  In particular, we show that together Theorems \ref{DM:slicing:lemma} and \ref{admissible:pair:DM:stack:conj} imply a form of Bombieri-Lang Conjecture for Deligne-Mumford stacks.  We refer to \cite[Remark 14.3.7]{Bombieri:Gubler} for an introduction to the classical formulation of the Bombieri-Lang Conjecture.  We also mention that a formulation of Vojta's Conjecture for general type Deligne-Mumford stacks has been proposed and studied in \cite{Ascher:Javanpeykar:2017}.  The approach used there builds on that of \cite{Abramovich:VarillyAlvarado:Pera:2017}.  A key point to the proof of Theorem \ref{DM:Bombieri-Lang} below is a generalized form of an inequality of Silverman from \cite{Silverman:1984}.  The idea is to bound the discriminant of points in projective space in terms of their heights.   We refer to Lemma \ref{height:function:lower:bound:discriminant:term} for more details.

\begin{theorem}\label{DM:Bombieri-Lang} 
Let $\mathcal{X}$ be a nonsingular Deligne-Mumford stack with projective coarse moduli space.  Let $\mathcal{D}_0$ be a normal crossings divisor on $\mathcal{X}$.  Assume that the log canonical bundle $\K_{(\mathcal{X},\mathcal{D}_0)} := \K_{\mathcal{X}} + \mathcal{D}_0$ is big.  Fix a finite surjective morphism $\pi \colon Y \rightarrow \mathcal{X}$ from a nonsingular projective $\K$-variety $Y$.  Fix $d \geq 1$ and let $r := (d \cdot \operatorname{deg}(\pi))[\KK:\QQ]$.  Assume that we can write
$\K_{(Y,D_0)}^{\otimes m} \simeq M^{\otimes n} \otimes N$
for some ample line bundle $M$ and effective line bundle $N$ and some positive integers $m,n>0$ which are such that
$n > m m_0(2r-2) \text{.}$
Here, $m_0 > 0$ is such that $M^{\otimes m_0}$ is very ample.  Then, Conjecture \ref{Vojta:Main:Conj:Disc:Bounded:Degree} implies that the set 
$$
\Gamma_{(\mathcal{X},\mathcal{D}_0,S),d} := \{ \text{$(\mathcal{D}_0,S)$-integral points $x \in \mathcal{X} \setminus \operatorname{Supp}(\mathcal{D}_0)$ : $[\KK(x):\KK] \leq d$}\}
$$
is not Zariski dense.
\end{theorem}

We prove Theorem \ref{DM:Bombieri-Lang} in Section \ref{Proof:DM:Bombieri-Lang}.  It is a consequence of Theorem \ref{lang:conjecture:bounded:degree}.  In Example \ref{lang:conjecture:bounded:degree:example}, we indicate that the hypothesis of Theorems \ref{DM:Bombieri-Lang} and \ref{lang:conjecture:bounded:degree} are indeed non-empty.  As a final comment, in Examples \ref{lang:conjecture:bounded:degree:example}, \ref{general:complete:intersections:example}, \ref{n:fold:Petri:example} and \ref{Extremal:curves:example} we illustrate the hypothesis of Theorem \ref{DM:Bombieri-Lang}, with emphasis on the case of nonsingular projective curves.  These examples are of some independent interest and we discuss how the conclusions drawn from Theorem \ref{DM:Bombieri-Lang}, in the form of Theorem \ref{lang:conjecture:bounded:degree}, related to the above mentioned results that pertain to the sets $\Gamma_{C,d}(\FF)$.  

\subsection*{Notation and other conventions} Throughout this article, $\KK$ denotes a number field with algebraic closure $\overline{\KK}$,  $M_{\KK}$ is its set of places, $M^0_{\KK}$ is its set of finite places and $M^{\infty}_{\KK}$ is its set of infinite places.  Our conventions about representatives $|\cdot|_v$ for places $v \in M_{\KK}$ follow those from \cite[p.~11]{Bombieri:Gubler}.   Henceforth, unless stated otherwise, we fix a finite subset $S \subset M_{\KK}$ that contains $M^{\infty}_{\KK}$.    We assume that all varieties, Deligne-Mumford stacks, divisors, line bundles and related objects are defined over $\KK$.  If $X$ is a nonsingular complete variety over $\KK$, then $\K_X$ denotes its canonical line bundle.

\subsection*{Acknowledgements} It is the author's pleasure to thank colleagues for their interest, correspondence and conversations on related topics.  Further, the author  thanks anonymous referees for their stimulating questions and thoughtful comments.  Finally, the author thanks the Natural Sciences and Engineering Research Council of Canada for their support through his grants, DGECR-2021-00218 and RGPIN-2021-03821, and also the National Science and Technology Council (Taiwan) for their support through his grants 115-2115-M-002-003-MY3 and 115-2811-M-002-043.

\section{Preliminaries}

\subsection{Discriminants}
Our conventions and notation for discriminants are as in \cite{Bombieri:Gubler}.  For example, the \emph{absolute logarithmic discriminant} $\mathrm{d}_{\KK}$ is defined to be
$$
\mathrm{d}_{\KK} := \frac{1}{[\KK:\QQ]} \log \left| D_{\KK / \QQ} \right| \text{.}
$$
Here, $D_{\KK / \QQ}$ is the \emph{discriminant} of $\KK / \QQ$.  We also put
$$
\mathrm{d}_{\KK,S} := - \frac{1}{[\KK:\QQ]} \sum\limits_{v \in M_{\KK} \setminus S} \log \left| D_{\KK_v / \QQ_p} \right|_p \text{.}
$$
So, in particular $\mathrm{d}_{\KK,M^{\infty}_{\KK}} = \mathrm{d}_{\KK}\text{.}$

Let $\FF / \KK$ be a finite extension with $\FF \subset \overline{\KK}$.  If $w$ is a finite place of $\FF$ that lies above a finite place $v$ of $\KK$, then $f_{w/v}$ denotes the \emph{residue degree} and $e_{w/v}$ denotes the \emph{ramification index}. With this notation
$e_{w/v} f_{w/v} = [\FF_w : \KK_v] \text{.}$

For later use, we recall that
$$
0 \leq \mathrm{d}_{\FF} - \mathrm{d}_{\KK} = - \sum_{v \in M^0_{\KK}} \sum_{w \mid v} \frac{[\FF_w : \KK_v]}{[\FF : \KK]}(e_{w/v} - 1 + \delta_w) \log | \pi_v |_v \text{.}
$$
(See \cite[Theorem B.2.12]{Bombieri:Gubler} and compare with \cite[Equation (14.7)]{Bombieri:Gubler}.)
Here $\delta_w \in [1, e_{w/v} v(e_{w/v})]$ if $v(e_{w/v}) \geq 1$ and zero otherwise. 
Further, setting
\begin{equation}\label{rel:discrim:eqn:3}
\mathrm{d}_{\FF,S} := - \frac{1}{[\FF:\QQ]} \sum_{v \in M_{\KK} \setminus S} \sum_{w \mid v} \log | D_{\KK_v / \QQ_p}|_p
\end{equation}
it then follows that 
\begin{equation}\label{rel:discrim:eqn:4}
\mathrm{d}_{\FF} - \mathrm{d}_{\KK} = \mathrm{d}_{\FF,S} - \mathrm{d}_{\KK,S} + \mathrm{O}(1)
\end{equation}
and
\begin{equation}\label{disc:diff:eqn:1}
\mathrm{d}_{\FF,S} - \mathrm{d}_{\KK,S} = - \sum_{v \in M^0_{\KK} \setminus S} \sum_{w \mid v} \frac{[\FF_w:\KK_v]}{[\FF:\KK]}(e_{w/v} - 1 + \delta_w) \log |\pi_v|_v \text{.}
\end{equation}

\subsection{Height, proximity and counting functions}\label{height:prox:count}

Our conventions follow \cite[Chapter 2]{Bombieri:Gubler} closely.  
Let $X$ be a geometrically irreducible projective variety over $\KK$.  If $x \in X(\overline{\KK}) \text{,}$ 
then let $\KK(x)$ be its field of definition and define its discriminant to be 
$\mathrm{d}_{\KK}(x) = \mathrm{d}_{\KK(x)} \text{.}$  Let $L = \Osh_X(D)$ be a line bundle on $X$ 
defined over $\KK$, with corresponding Cartier divisor $D$.  Fix a \emph{presentation} for $D$ to get a local Weil function $\lambda_D(\cdot,v)$ for each $v \in M_{\KK}$, and defined over $\KK$.   
Define the \emph{counting function} 
$$
n_{S}(D,x) := \sum_{w \mid v \in M_{\KK} \setminus S} \lambda_D(x,w) \text{ for $x \in X \setminus \operatorname{Supp}(D)$}
$$
and the \emph{proximity function}
$$
m_{S}(D,x) := \sum_{w \mid v \in S} \lambda_D(x,w) \text{ for $x \in X \setminus \operatorname{Supp}(D)$.}
$$
Denoting by $h_L(\cdot)$ the \emph{absolute logarithmic height} of $L$, \cite[Section 2.3.3]{Bombieri:Gubler}, then
\begin{equation}\label{height:proximity:counting:eqn}
h_L(x) = m_{S}(D,x) + n_{S}(D,x) + \mathrm{O}(1) \text{ for $x  \in X \setminus \operatorname{Supp}(D)$.}
\end{equation}
   
In order to define a truncated counting function, we first put 
$$
\chi(t) := \begin{cases}
0 & \text{ if $t \leq 0$} \\
1 & \text{ if $t > 0$.}
\end{cases}
$$
Then the \emph{truncated counting function} is defined to be
\begin{equation}\label{truncated:counting:function:defn}
n^{(1)}_{S}(D,x) := \sum_{w | v \in M_{\KK} \setminus S} \chi(\lambda_{D}(x,w)) \log \left| \frac{1}{\pi_w} \right|_w \text{.}
\end{equation}
(Compare with \cite[Definition 14.2.2]{Bombieri:Gubler}, \cite{Vojta:1998} and \cite{Abramovich:VarillyAlvarado:Pera:2017}.)
In \eqref{truncated:counting:function:defn}, $\pi_w$ is a local parameter of $w$.   

Observe that here we require that the divisor $D$ is defined over $\KK$.  This is a less general context than considered, for instance, in \cite{Grieve:points:bounded:degree}, \cite{Grieve:qualitative:subspace}, \cite{Grieve:Noytaptim:coord:size} or \cite{Grieve:Effective:Local:Weil}.  However, the context that we pursue here is well suited for the setting of integral points with respect to a normal crossings divisor and finite set of places.  Indeed, the reason is that it allows for the interplay amongst the height function, the proximity and counting functions along the lines of the relation \eqref{height:proximity:counting:eqn}.  

In this sense the techniques employed in the present article are limited in so far as that they do not allow, for example as in \cite{Grieve:points:bounded:degree}, \cite{Grieve:qualitative:subspace}, \cite{Grieve:Noytaptim:coord:size} or \cite{Grieve:Effective:Local:Weil}, for the case of proximity and counting functions for divisors which are defined over a finite extension of the base number field and with respect to finite sets of places of the given base number field.  Finally, for the case of complete varieties over $\KK$, height, proximity and counting functions for Cartier divisors can be defined via this framework of presentations, via pullback, as an application of Chow's lemma.

\subsection{Logarithmic pairs}

Our conventions are similar to those of \cite{Kollar:Mori:1998} and others.   Let $X$ be an $n$-dimensional variety over $\KK$.  A  \emph{strict normal crossings divisor} on $X$, and defined over $\KK$, is an effective Cartier divisor $D \subset X\text{,}$ 
defined over $\KK$, which is such that for all $x \in D$, the local ring $\Osh_{X,x}$ is regular and there exists a  regular  system of parameters $x_1,\dots,x_n \in \mathfrak{m}_x$ together with an integer $1 \leq r \leq n$ which are such that $D$ is cut out by $x_1,\dots,x_r$ in $\Osh_{X,x}$.
A \emph{normal crossings divisor} on $X$, and defined over $\KK$, is an effective Cartier divisor $D \subset X$ 
which is defined over $\KK$ and which is such that for all $x \in D\text{,}$ 
there exists an \'etale morphism $U \rightarrow X$ 
with $x$ in its image and  with the property that $D \times_X U \subset U$ is a strict normal crossings divisor.
Finally, we recall the concept of \emph{logarithmic resolution of singularities}.  To this end, let $X$ be a normal $\KK$-variety and let $D$ be a $\QQ$-Weil divisor on $X$.  As in \cite[p. 5]{Kollar:Mori:1998}, by a \emph{log resolution} of the pair $(X,D)$ is meant a proper birational morphism $g \colon Y \rightarrow X$ 
from a nonsingular $\KK$-variety $Y$ which is such  (i)  the exceptional set $\operatorname{Ex}(g)$ is a divisor; and (ii) 
$\operatorname{Ex}(g) \bigcup g^{-1}(\operatorname{Supp} D)$ is a strict normal crossings divisor on $Y$.  
That log resolutions exist follows as a a consequence of Hironaka's Theorem about resolutions of singularities \cite[Theorem 0.2]{Kollar:Mori:1998}.

\section{Ramified covers and inequalities for truncated counting functions}

In this section, we adapt from \cite[Section 3]{Vojta:1998} various points which are needed to establish that Conjecture \ref{Admissible:pairs:Vojta:Main:Conj:Disc:Bounded:Degree} implies Conjecture \ref{truncated:Admissible:pairs:Vojta:Main:Conj:Disc:Bounded:Degree}.  For ease of reference and also since our approach to heights, local Weil, counting and proximity functions, as described in Section \ref{height:prox:count}, are different compared to those from \cite{Vojta:1998}, we state explicitly and prove the results that we require here.  In doing so, we indicate clearly the manner in which they resemble the approach from \cite{Vojta:1998}.

To begin with, in Lemma \ref{lemma:3:2}, we state \cite[Lemma 3.2]{Vojta:1998}.  It describes the manner in which log canonical divisors behave under pullback via generically finite morphisms. 

\begin{lemma}
\label{lemma:3:2}
Working over $\KK$, suppose that $\pi \colon X' \rightarrow X$ 
is a generically finite morphism of nonsingular varieties.  Respectively, let $D$ and $D'$ be normal crossings divisors on $X$ and $X'$ which are such that $\operatorname{Supp} D' = (\pi^* D)_{\mathrm{red}} \text{.}$ Then the line bundle $\K_{X'} \otimes \Osh_{X'}(D') \otimes \pi^*(\K_X^{\vee} \otimes \Osh_X(-D))$ is effective and admits a nonzero global section that 
that vanishes only on $\operatorname{Supp}(D') \bigcup \operatorname{Ram}(\pi)$.
\end{lemma}
\begin{proof}
This is \cite[Lemma 3.2]{Vojta:1998}.  
\end{proof}

Next, we establish a reduction step that is similar to \cite[p. 1112]{Vojta:1998}.

\begin{lemma}
\label{Vojta:Chow}
In proving that Conjecture \ref{Admissible:pairs:Vojta:Main:Conj:Disc:Bounded:Degree} implies Conjecture \ref{truncated:Admissible:pairs:Vojta:Main:Conj:Disc:Bounded:Degree}, there is no loss in generality by assuming that $X$ is a nonsingular projective variety.
\end{lemma}

\begin{proof}  We argue similar to the discussion given in \cite[p. 1112]{Vojta:1998}.
By Chow's Lemma and resolution of singularities there exists a smooth projective variety $X'$ and a proper birational morphism $\pi \colon X' \rightarrow X$ which is such that $D' := (\pi^* D)_{\mathrm{red}}$ 
has normal crossings.  
In fact, we can write
$D_0' + D_1' := D'$
where $D_0'$ is a normal crossings divisor on $X'$ and where $D_1'$ is an effective divisor on $X'$.
Thus, in light of Lemma \ref{lemma:3:2}, it follows that Conjecture \ref{truncated:Admissible:pairs:Vojta:Main:Conj:Disc:Bounded:Degree} applied to $(X',D_1')$ and with respect to $(D_0',S)$-integral points implies Conjecture \ref{truncated:Admissible:pairs:Vojta:Main:Conj:Disc:Bounded:Degree} for $(X,D_1)$ with respect to $(D_0,S)$-integral points.  
\end{proof}

In deducing Conjecture \ref{Admissible:pairs:Vojta:Main:Conj:Disc:Bounded:Degree} from Conjecture \ref{truncated:Admissible:pairs:Vojta:Main:Conj:Disc:Bounded:Degree}, it is helpful to replace the divisor $D_1$ by a sufficiently positive multiple.  This is the role of Lemma \ref{reduction:step:2}.

\begin{lemma}
\label{reduction:step:2}
Let $D_1$ and $D_2$ be two effective divisors on a nonsingular projective variety $X$.  Then 
$$
n^{(1)}_S(D_1+D_2,x) - n_S^{(1)}(D_1,x) \\ \leq  h_{\Osh_X(D_2)}(x) + \mathrm{O}(1) 
$$
for $x \in X \setminus \operatorname{Supp}(D_1 + D_2)$.
\end{lemma}
 
\begin{proof}  We argue as in \cite[pp. 1112--1113]{Vojta:1998}.
Using the definition of the truncated counting function, \eqref{truncated:counting:function:defn}, it follows that 
\begin{multline*}
n^{(1)}_S(D_1+D_2,x) - n_S^{(1)}(D_1,x)  \leq n_S^{(1)}(D_2,x) \\
 \leq n_S(D_2,x) 
 \leq h_{\Osh_X(D_2)}(x) + \mathrm{O}(1) \text{.}
\end{multline*}
Thus, \eqref{eqn:6} with $D_1$ replaced by $D_1 + D_2$ implies \eqref{eqn:6} applied to $D_1$.
\end{proof}

Lemma \ref{reduction:step:2} allows for the following reduction step.

\begin{lemma}
\label{reduction:step:2:prime} 
Let $X$ be a nonsingular projective variety and let $D_0$ be a normal crossings divisor on $X$.  Suppose that $D_1$ and $D_2$ are effective divisors on $X$ which are such that $D_0+D_1$ and $D_0 + D_1 + D_2$ have normal crossings.  Then the conclusion of Conjecture \ref{truncated:Admissible:pairs:Vojta:Main:Conj:Disc:Bounded:Degree} for $D_1+D_2$ implies that the conclusion of Conjecture \ref{truncated:Admissible:pairs:Vojta:Main:Conj:Disc:Bounded:Degree} also holds for $D_1$.  
\end{lemma}

\begin{proof}  The proof follows by adapting the discussion of \cite[p.~1113]{Vojta:1998}.
Applying Lemma \ref{reduction:step:2}, we deduce that
$$n^{(1)}_S(D_1+D_2,x) - n_S^{(1)}(D_1,x) \leq  h_{\Osh_X(D_2)}(x) + \mathrm{O}(1)$$
for points $x \in X \setminus \operatorname{Supp}(D_1+D_2)\text{.}$
Thus, Conjecture \ref{truncated:Admissible:pairs:Vojta:Main:Conj:Disc:Bounded:Degree} for $D_1+D_2$ implies that the conclusion of Conjecture \ref{truncated:Admissible:pairs:Vojta:Main:Conj:Disc:Bounded:Degree} also holds for $D_1$.  
\end{proof}

In deducing Conjecture \ref{truncated:Admissible:pairs:Vojta:Main:Conj:Disc:Bounded:Degree} from Conjecture \ref{Admissible:pairs:Vojta:Main:Conj:Disc:Bounded:Degree}, there is no loss in generality by assuming that $D_1$ is very ample.  Indeed, this is a consequence of Lemma \ref{reduction:step:2:prime} which is used to establish Lemma \ref{reduction:step:3} below.

\begin{lemma}
\label{reduction:step:3}
In order to prove that Conjecture \ref{Admissible:pairs:Vojta:Main:Conj:Disc:Bounded:Degree} implies Conjecture \ref{truncated:Admissible:pairs:Vojta:Main:Conj:Disc:Bounded:Degree}, there is no loss in generality by assuming that $X$ is a nonsingular projective variety and  that $D_1$ is a very ample divisor on $X$ and having the property that $D_0+D_1$ is a normal crossings divisor. 
\end{lemma}

\begin{proof}  We argue similar to \cite[p.~1113]{Vojta:1998}.
Choose a sufficiently very ample divisor $D_2$ that is transverse to $D_1$.  Then upon applying Lemma \ref{reduction:step:2:prime} to $D_1+D_2$, we may replace $D_1$ by the very ample divisor $D_1 + D_2$ to obtain the conclusion that is desired by Lemma  \ref{reduction:step:2:prime}.
\end{proof}

The ramified cover that gets used to deduce Conjecture \ref{truncated:Admissible:pairs:Vojta:Main:Conj:Disc:Bounded:Degree} from Conjecture \ref{Admissible:pairs:Vojta:Main:Conj:Disc:Bounded:Degree} is constructed in Lemma \ref{lemma:3:4}.  This lemma combines \cite[Lemmas 3.3 and 3.4]{Vojta:1998}.

\begin{lemma}
\label{lemma:3:4}
Let $A$ be a very ample normal crossings divisor on a nonsingular $n$-dimensional projective variety $X$.  Then, there exists a normal crossings divisor $D^*$ on $X$ which is such that
\begin{enumerate}
\item[(i)]{$D^* - A \geq 0$;}
\item[(ii)]{for all $e \in \ZZ_{>0}$, there exists a smooth projective variety $X'$ and a proper generically finite morphism $\pi \colon X' \rightarrow X$ which is such that $(\pi^* D^*)_{\mathrm{red}}$ has normal crossings and all components of $\pi^* A$ have multiplicity $\geq e$ or zero;}
\item[(iii)]{for all $e \in \ZZ_{>0}$, the morphism $\pi \colon X' \rightarrow X$ in (ii) is unramified outside of $\pi^{-1}(\operatorname{Supp} D^*)$;}
\item[(iv)]{for all $e \in \ZZ_{>0}$, the function field of the variety $X'$, defined in (ii), is described as $\KK(X') = \KK(X)(\sqrt[e]{f_{1}},\dots,\sqrt[e]{f_{n}})$ for some $f_{1},\dots,f_{n} \in \KK(X)^{\times}$; and}
\item[(v)]{for all $e \in \ZZ_{>0}$, the divisor $\pi^*(D^*+ A)_{\mathrm{red}}$, for $\pi$ the morphism defined in (ii), is a normal crossings divisor on $X'$.}
\end{enumerate}
\end{lemma}

\begin{proof}  We combine the approach of \cite[Lemmas 3.3 and 3.4]{Vojta:1998}.
By Bertini's theorem, there are effective divisors $D_{1},\dots, D_{n}$ which are such that
$D_{i} \sim A$, for $j = 1,\dots,n$, and such that
$D^* := A + D_{1}+\hdots+D_{n}$ has normal crossings and
$ A \bigcap D_{1} \bigcap \hdots \bigcap D_{n} = \emptyset$.

Fix such a collection of divisors and let $e \in \ZZ_{>0}$.  For each $i=1,\dots,n$, choose $f_i \in \KK(X)^{\times}$ which is such that $\operatorname{div}(f_i) = A - D_{i}\text{.}$  
Let $\pi_i \colon X_{i} \rightarrow X$ 
be the composition of the normalization of $X$ in $\KK(X)\left(\sqrt[e]{f_{i}}\right)$ together with a desingularization of this normalization which has the property that $(\pi_{i}^* A + D_{i})_{\mathrm{red}}$ has normal crossings.

Let $\pi \colon X' \rightarrow X$ be a desingularization of the normalization of $X$ in the compositum $\KK(X') = \KK(X_{1})\cdot\hdots\cdot\KK(X_{n})$ and which is such that $(\pi^* D^*)_{\mathrm{red}}$ has normal crossings; and $X'$ dominates $X_{1},\dots,X_{n}$.  We may assume that $\pi$ is \'etale outside of $\pi^{-1}(\operatorname{Supp} D^*)$.  

Now, suppose that $E$ is a component of $\pi^* D^*$.  Then, there exists some $i$, $1 \leq i \leq n$, with $\pi(E) \subseteq A \bigcup D_{i}$; and $\pi(E) \not \subseteq A \bigcap D_{i}$.  But then since all components of $\pi^*_i(A)$ not lying over $D_i$ have multiplicity some positive multiple of $e$, it follows that the image of $E$ in $X_{i}$ is contained in a component of $\pi^{-1}(A+D_i)$ that has multiplicity of at least equal to $e$.
\end{proof}

Finally, in Lemma \ref{lemma:3:5}, we make precise the manner in which the discriminant transforms under suitable ramified covers.

\begin{lemma}
\label{lemma:3:5}
Let $D_0$ be a normal crossings divisor on a projective variety $X$.  Let $D_1$ be a very ample divisor on $X$ which has the property that $D_0+D_1$ is a normal crossings divisor on $X$.  Let $\pi \colon X' \rightarrow X$ 
be a generically finite proper morphism from a nonsingular projective variety $X'$ and assume that 
$D_1':= (\pi^* D_1)_{\mathrm{red}}$ 
has normal crossings on $X'$, that all components of $\pi^*D_1$ have multiplicity at least equal to $e$, that $\pi \colon X' \rightarrow X$ 
is unramified outside of $D' := D_0'+D_1'$ where $D'_0 := (\pi^* D_0)_{\mathrm{red}}$ 
is a normal crossings divisor on $X'$.  In this setting, assume that $S$ is such that given $d \in \ZZ_{>0}$ 
\begin{equation}\label{lemma:3:5:eqn:points}
\text{for all
$x \in X(\overline{\KK}) \setminus \operatorname{Supp}D$ 
with $[\KK(x):\KK] \leq d$ and all points $x' \in \pi^{-1}(x)$}
\end{equation} 
 it holds true that 
\begin{equation}\label{lemma:3:5:eqn:points:places:ramification:condition}
\mathrm{d}_{\KK,S}(x') - \mathrm{d}_{\KK,S}(x) \leq (e-1) n^{(1)}_S(D',x') \text{.}
\end{equation}
Then, for all such points \eqref{lemma:3:5:eqn:points} it also holds true that 
$$
\mathrm{d}_{\KK}(x') - \mathrm{d}_{\KK}(x) \leq n_S^{(1)}(D_1,x) - n_S(D'_1,x') + \frac{1}{e} h_{\Osh_X(D_1)}(x) + \mathrm{O}(1) \text{.}
$$
\end{lemma}
\begin{proof}  We adapt the proof of \cite[Lemma 3.5]{Vojta:1998} to our current context.
Using \eqref{lemma:3:5:eqn:points:places:ramification:condition} in conjunction with \eqref{rel:discrim:eqn:4} we can write
\begin{equation}\label{eqn:1}
\mathrm{d}_{\KK}(x') - \mathrm{d}_{\KK}(x) \leq (e-1) n^{(1)}_S(D'_1,x') + \mathrm{O}(1) \text{.}
\end{equation}

Now, by assumption, all components of $\pi^*D_1$ have multiplicity at least equal to $e$.  It then follows, using the definition of the truncated counting functions, that 
\begin{equation}\label{eqn:2}
n^{(1)}_S(D_1,x) - n^{(1)}_S(D'_1,x') \geq (e-1) n^{(1)}_S(D'_1,x') 
\end{equation}
and
\begin{multline}\label{eqn:3}
n_S(D'_1,x') - n^{(1)}_S(D'_1,x') \leq n_S(D'_1,x') \\ \leq \frac{1}{e} n_S(D_1,x) 
\leq \frac{1}{e} h_{\Osh_X(D_1)}(x) + \mathrm{O}(1) \text{.}
\end{multline}
Combining the inequalities \eqref{eqn:2} and \eqref{eqn:3} gives the inequality
\begin{multline}\label{eqn:4}
n^{(1)}_S(D_1,x) - n_S(D'_1,x') + \frac{1}{e} h_{\Osh_X(D_1)}(x) 
\geq  (e-1)n^{(1)}_S(D'_1,x') + \mathrm{O}(1) \text{.}
\end{multline}
Finally, together the inequalities \eqref{eqn:4} and \eqref{eqn:1} imply that 
$$
\mathrm{d}_{\KK}(x') - \mathrm{d}_{\KK}(x) \leq n^{(1)}_S(D_1,x) - n_S(D'_1,x') + \frac{1}{e} h_{\Osh_X(D_1)}(x) + \mathrm{O}(1) \text{.}
$$
\end{proof}

\begin{remark}\label{Remark:Hermite:discriminant:thm}
The condition given by \eqref{lemma:3:5:eqn:points:places:ramification:condition} can be satisfied for suitably large sets of places $S$.  Indeed, by Hermite's discriminant theorem, in the form of \cite[Corollary B.2.15]{Bombieri:Gubler},  it is possible to choose a smallest such $S$, which may a priori be larger than the given $S$, so that this condition is satisfied.  Indeed, the existence of such a smallest $S$ can be deduced using \eqref{disc:diff:eqn:1} and \cite[Corollary B.2.15]{Bombieri:Gubler}. (Here, is where we use the fact that we are working over a number field as opposed to the function field of a curve.  But, by modifying our approach here, that setting can be treated as in \cite[Lemma 3.5]{Vojta:1998}.)
\end{remark}

\section{Proof of Theorem \ref{Vojta:logical:implications}}\label{Section:proof:harder:implication}

We first check, similar to \cite[Lemma 2.3]{Levin:GCD}, that Conjecture \ref{Vojta:Main:Conj:Disc:Bounded:Degree} is logically equivalent to Conjecture \ref{Admissible:pairs:Vojta:Main:Conj:Disc:Bounded:Degree}.  

\begin{lemma}\label{Levin:lemma:2:3}
Conjecture \ref{Vojta:Main:Conj:Disc:Bounded:Degree} is logically equivalent to Conjecture \ref{Admissible:pairs:Vojta:Main:Conj:Disc:Bounded:Degree}.
\end{lemma}

\begin{proof} 
Consider a set of $(D_0,S)$-integral points \eqref{eqn:5}.  Then, by definition 
$$
n_S(D_0,x) = \mathrm{O}(1)
$$
for all $x \in R$.  On the other hand, by linearity of the counting function, 
it follows that
$$
n_S(D_1+D_0,x) = n_S(D_1,x) + n_S(D_0,x) + \mathrm{O}(1) \text{.}
$$
Therefore the conclusion desired by Conjecture \ref{Admissible:pairs:Vojta:Main:Conj:Disc:Bounded:Degree} follows from the conclusion of Conjecture \ref{Vojta:Main:Conj:Disc:Bounded:Degree} applied to the normal crossings divisor $D_1 + D_0$.

To see that Conjecture \ref{Admissible:pairs:Vojta:Main:Conj:Disc:Bounded:Degree} implies Conjecture \ref{Vojta:Main:Conj:Disc:Bounded:Degree}, note that when $D_0 = 0$, then Conjecture \ref{Admissible:pairs:Vojta:Main:Conj:Disc:Bounded:Degree} becomes the statement of Conjecture \ref{Vojta:Main:Conj:Disc:Bounded:Degree}.
\end{proof}

Next, we check that Conjecture \ref{truncated:Admissible:pairs:Vojta:Main:Conj:Disc:Bounded:Degree} implies Conjecture \ref{Admissible:pairs:Vojta:Main:Conj:Disc:Bounded:Degree}.

\begin{lemma}\label{admissible:pairs:easy:implication}
Conjecture \ref{truncated:Admissible:pairs:Vojta:Main:Conj:Disc:Bounded:Degree} implies Conjecture \ref{Admissible:pairs:Vojta:Main:Conj:Disc:Bounded:Degree}.
\end{lemma}

\begin{proof}
Indeed, simply note that $n_S^{(1)}(D_1,x) \leq n_S(D_1,x)$.
\end{proof}

We are now in a position to prove Theorem \ref{Vojta:logical:implications}.  Our argument follows the approach of \cite[pp.~1115--1116]{Vojta:1998}.

\begin{proof}[Proof of Theorem \ref{Vojta:logical:implications}]  
In light of Lemmas \ref{Levin:lemma:2:3} and \ref{admissible:pairs:easy:implication}, it remains to establish that Conjecture \ref{Admissible:pairs:Vojta:Main:Conj:Disc:Bounded:Degree} implies Conjecture \ref{truncated:Admissible:pairs:Vojta:Main:Conj:Disc:Bounded:Degree}.  By Lemma \ref{Vojta:Chow} we may assume that $X$ is projective.

First of all, by Kodaira's lemma, after adjusting $\epsilon > 0$, if required, we may assume that $L$ is ample.  Then
$$
h_{\Osh_X(D_1)}(\cdot) \leq c h_L(\cdot) +\mathrm{O}(1)
$$
for some constant $c$ that depends only  on $X$, $D$ and $L$.  Further, by Lemma \ref{reduction:step:3}, we may and do assume that $D_1$ is a very ample normal crossings divisor.  Now, fix an integer 
$e \geq c/\epsilon$
and let
$\pi \colon X' \rightarrow X$
be a generically finite cover of $X$ as  in the above Lemma \ref{lemma:3:4}.  

Recall, that $X'$ is a projective nonsingular variety,  $D'_1 := (\pi^* D_1)_{\mathrm{red}}$ 
has normal crossings, that all components of $\pi^* D_1$ have multiplicity $\geq e$ and that  $\pi$ is unramified outside of $\pi^{-1}(\operatorname{Supp} D_1)$.  Further, if $D_0' = (\pi^* D_0)_{\mathrm{red}} \text{,}$ 
then $D_0'+D_1'$ has normal crossings.

In light of Lemma \ref{lemma:3:5}, by enlarging $S$ if required, we may assume that given $d  \in \ZZ_{>0}$, if
$x \in X(\overline{\KK}) \setminus \operatorname{Supp}(D_1)$
has 
$
[\KK(x) : \KK] \leq d \text{,}
$
then for all 
$
x' \in \pi^{-1}(x)
$
it holds true that
$$
\mathrm{d}_{\KK}(x') - \mathrm{d}_{\KK}(x) \leq   n^{(1)}_S(D_1,x) - n_S(D'_1,x') + \frac{1}{e} h_{\Osh_X(D_1)}(x) + \mathrm{O}(1) \text{.}
$$

Note that bounded degree points $x \in X(\overline{\KK}) \setminus \operatorname{Supp} (D_0+D_1)$
lift to bounded degree points $x' \in X'(\overline{\KK})$. Put  $L' = \pi^* L \text{.}$

Now, Conjecture \ref{Admissible:pairs:Vojta:Main:Conj:Disc:Bounded:Degree} gives  the inequality
\begin{equation}\label{eqn:3:7}
n_S(D'_1,x')+\mathrm{d}_{\KK}(x') \geq h_{\K_{X'}+D'_0 + D_1'}(x') - \epsilon' h_{L'}(x') - \mathrm{O}(1)
\end{equation}
provided that 
$x' \not  \in Z'$
for  
$Z'\subsetneq  X$ 
some proper Zariski closed subset of $X$.  On the other hand, recall,  that
$$
\K_{X'} + D'_0 + D'_1 - \pi^*(\K_X + D_0 + D_1) \geq  0 \text{.}
$$
So, it follows that if 
$x' \in X'(\overline{\KK}) \setminus \operatorname{Supp}(D'_0+D_1')$
and $x = \pi(x')$
then
\begin{equation}\label{eqn:star}
h_{\K_{X'}+D'_0+D_1'}(x') \geq  h_{\K_X + D_0+D_1}(x) - \mathrm{O}(1).
\end{equation}

Now, by Lemma \ref{lemma:3:5}
\begin{equation}\label{eqn:*}
n_S(D'_1,x') + \mathrm{d}_{\KK}(x') \leq  n^{(1)}_S(D_1,x) + \mathrm{d}_{\KK}(x) + \frac{1}{e} h_{\Osh(D_1)}(x) + \mathrm{O}(1)
\end{equation}
for all $x' \in  X'(\overline{\KK}) \setminus \operatorname{Supp}(D_0'+D_1')$ where $x = \pi(x')$.

Henceforth, choose $\epsilon' > 0$ so  that
$\epsilon' <  \frac{\epsilon c}{e} \text{.}$
Then, using the inequality
$$
h_{\Osh_X(D_1)}(\cdot)\leq c h_L(\cdot) + \mathrm{O}(1) 
$$
it follows that
$$
\frac{1}{e} h_{\Osh_X(D_1)}(x) + \epsilon' h_{L'}(x') \leq \epsilon h_L(x)+ \mathrm{O}(1) \text{.}
$$
This inequality can be rewritten as
\begin{equation}\label{eqn:**}
- \frac{1}{e} h_{\Osh_X(D_1)}(x) - \epsilon' h_{L'}(x') - \mathrm{O}(1) \geq - \epsilon h_L(x) \text{.}
\end{equation}

Together, \eqref{eqn:3:7} and \eqref{eqn:*} give the inequality
\begin{multline}\label{eqn:***}
h_{\K_{X'} +D_0'+D_1'}(x') - \epsilon' h_{L'}(x') - \mathrm{O}(1) \\
\leq n_S(D_1',x') + d_{\KK}(x') 
\\
\leq n_S^{(1)}(D_1,x) + \mathrm{d}_{\KK}(x) + \frac{1}{e} h_{\Osh_X(D_1)}(x) + \mathrm{O}(1) 
\end{multline}
and it follows from \eqref{eqn:***} that
\begin{multline}\label{harder:key:end:eqn1}
n_S^{(1)}(D_1,x) + \mathrm{d}_{\KK}(x) + \frac{1}{e} h_{\Osh_X(D_1)}(x) + \mathrm{O}(1)\\
 \geq h_{\K_{X'}+D_0'+D_1'}(x') - \epsilon' h_{L'}(x') - \mathrm{O}(1) \text{.}
\end{multline}

On the other hand, in light of \eqref{eqn:star} and \eqref{eqn:**}, \eqref{harder:key:end:eqn1} can be rewritten as 
$$
n_S^{(1)}(D_1,x) + \mathrm{d}_{\KK}(x) \geq h_{\KK_X + D_0+D_1}(x) - \epsilon h_L(x) - \mathrm{O}(1)
$$
for all bounded degree $(D_0,S)$-integral points $x \in X(\overline{\KK}) \setminus \operatorname{Supp}(D_0+ D_1)$
outside of some proper Zariski closed subset $Z \subsetneq X \text{.}$
\end{proof}

\section{Proof of Theorem \ref{DM:Bombieri-Lang}}\label{Proof:DM:Bombieri-Lang}

In this section, we prove Theorem \ref{DM:Bombieri-Lang}.  The main input is Lemma \ref{height:function:lower:bound:discriminant:term} which builds on \cite[Theorem 2]{Silverman:1984}.    

\begin{lemma}\label{height:function:lower:bound:discriminant:term}
Let $M$ be an ample line bundle on a projective variety $Y$ and defined over a number field $\KK$.  Let $m_0 \in \ZZ_{>0}$ be a positive integer which is such that $M^{\otimes m_0}$ is very ample.  Let $y \in Y(\overline{\KK})$ be such that $[\KK(y):\QQ] \leq r$.  Then
\begin{equation}\label{discrim:eqn:inequality}
\mathrm{d}_{\KK}(y) = \mathrm{d}_{\KK(y)} \leq (2r-2)m_0 h_M(y) + \mathrm{O}(1) \text{.}
\end{equation}
\end{lemma}

\begin{proof}
The case that $Y = \PP^n$ and $M = \Osh_{\PP^n}(1)$ follows upon arguing either as in \cite[Proof of Theorem 2]{Silverman:1984} or deducing the result following the approach of \cite[Proposition 1.6.9]{Bombieri:Gubler}.  The more general case then follows in light of the relation that 
$$
h_{M^{\otimes m_0}}(x) = m_0 h_M(x) + \mathrm{O}(1) \text{.}
$$
\end{proof}

\begin{remark}
While we expect Lemma \ref{discrim:eqn:inequality} not to be optimal, in general, we do remark that in many ways it can be seen as an effectively computable upper bound for the discriminant.
\end{remark}

Theorem \ref{lang:conjecture:bounded:degree} is a special case of Theorem \ref{DM:Bombieri-Lang}.  But, in fact, we deduce Theorem \ref{DM:Bombieri-Lang} from Theorem \ref{lang:conjecture:bounded:degree}.

\begin{theorem}\label{lang:conjecture:bounded:degree}  Fix a finite set of places $S \subset M_{\KK}$ that contains all archimedean places and assume that Conjecture \ref{Admissible:pairs:Vojta:Main:Conj:Disc:Bounded:Degree} holds true.  Let $(Y,D_0)$ be a projective normal crossings pair, defined over $\KK$, with the property that the log canonical divisor $\K_{(Y,D_0)} := \K_Y + D_0$ is big.  Fix an ample line bundle $M$ on $Y$ together with a positive integer $m_0>0$ which is such that the line bundle $M^{\otimes m_0}$ is very ample.  Fix a positive integer $r > 0$ and assume that we can write
\begin{equation}\label{lang:conjecture:bounded:degree:soltn:set:key:line:bundle:relation}
\K_{(Y,D_0)}^{\otimes m} \simeq M^{\otimes n} \otimes N
\end{equation}
for some effective line bundle $N$ and some positive integers $m,n>0$ which are such that 
\begin{equation}\label{lang:conjecture:bounded:degree:soltn:set:key:inequality}
n > m m_0(2r-2) \text{.}
\end{equation}
Then the set
\begin{equation}\label{bounded:degree:integral:points}
\Gamma_{(Y,D_0;S),r}(\KK) := \{\text{$(D_0,S)$-integral points $y \in Y \setminus \operatorname{Supp} D_0$: $[\KK(y):\QQ] \leq r$} \}
\end{equation}
is not Zariski dense.
\end{theorem}

\begin{proof}
Let $y \in Y(\overline{\KK})$ have the property that $[\KK(y):\QQ] \leq r$.  Then by Lemma \ref{height:function:lower:bound:discriminant:term},
\begin{equation}\label{lang:conjecture:bounded:degree:soltn:set:eqn1}
\mathrm{d}_{\KK}(y) = \mathrm{d}_{\KK(y)} \leq (2r-2)m_0 h_M(y) + \mathrm{O}(1) \text{.}
\end{equation}

Let $\epsilon > 0$ and set $\epsilon' = \epsilon / m$.  By assumption, we may write
\begin{equation}\label{canonical:bundle:formula:eqn}
\K_{(Y,D_0)}^{\otimes m} \simeq M^{\otimes n} \otimes N
\end{equation}
for some effective line bundle $N$ and some positive integers $m,n>0$ which are such that $n > m(2r-2)m_0$.

Conjecture \ref{Vojta:Main:Conj:Disc:Bounded:Degree}, in the form of Conjecture \ref{Admissible:pairs:Vojta:Main:Conj:Disc:Bounded:Degree}, applied to $\epsilon'$, $M^{\otimes n}$ and with respect to the divisor $D_1 = 0$ yields the inequality that 
\begin{equation}\label{lang:conjecture:bounded:degree:soltn:set:eqn2}
h_{\K_{(Y,D_0)}}(y) - \frac{\epsilon}{m} h_{M^{\otimes n}}(y) \leq \mathrm{d}_{\KK}(y) + \mathrm{O}(1)
\end{equation}
for all $(D_0,S)$-integral points $y \in Y \setminus Z$ with $[\KK(x):\KK] \leq d$.  Here, $Z \subsetneq Y$ is a proper Zariski closed subset and $d = r / [\KK:\QQ]$.

Now, using \eqref{canonical:bundle:formula:eqn} and \eqref{lang:conjecture:bounded:degree:soltn:set:eqn1}, \eqref{lang:conjecture:bounded:degree:soltn:set:eqn2} implies that
\begin{equation}\label{lang:conjecture:bounded:degree:soltn:set:eqn3}
(n(1-\epsilon) - m m_0 (2r - 2))h_M(x) \leq \mathrm{O}(1) \text{.}
\end{equation}
 In \eqref{lang:conjecture:bounded:degree:soltn:set:eqn3}, by ensuring that $\epsilon > 0$ is sufficiently small, e.g., $0 < \epsilon \ll 1$, we have that 
 $$
 n(1-\epsilon) - m m_0 (2r-2) > 0 \text{.}
 $$
 So, the Northcott property, \cite[Theorem 2.4.9]{Bombieri:Gubler}, implies that there are at most finitely many solutions to \eqref{lang:conjecture:bounded:degree:soltn:set:eqn3}.    The non-Zariski density of the set \eqref{bounded:degree:integral:points} then follows.
\end{proof}

In the following Examples \ref{lang:conjecture:bounded:degree:example}, \ref{general:complete:intersections:example}, \ref{n:fold:Petri:example} and \ref{Extremal:curves:example}, we show that the hypothesis \eqref{canonical:bundle:formula:eqn} is indeed not empty; we also make some remarks that pertain to the inequality \eqref{lang:conjecture:bounded:degree:soltn:set:key:inequality} specifically.  We place particular emphasis on the case of dimension $1$ and discuss connections to the concept of \emph{arithmetic degree of irrationality} from \cite{Smith:Vogt:2022}.

\begin{example}\label{lang:conjecture:bounded:degree:example}
Consider the simplest non-trivial instance of the above inequality \eqref{lang:conjecture:bounded:degree:soltn:set:key:inequality}, namely the case that $r = 2$, $m_0 = 1$, $m=1$ and $n = 3$.  Then, we want to write 
$\K_Y \simeq M^{\otimes 3} \otimes \Osh_Y(N)$
with $M$ very ample.  So, for example, consider the case of a non-singular degree $d$ and genus $g$ plane curve $C \subset \PP^2_{\QQ} \text{.}$  Then 
$\K_C \simeq \Osh_C(d-3)$
and
$g = \frac{1}{2}(d-1)(d-2) \text{.}$
Thus, if $d = 6$, then $g = 10$ and
$\K_C \simeq \Osh_C(3) \simeq H^{\otimes 3}$
for $H = \Osh_C(1)$ a hyperplane class.  
\end{example}

\begin{example}\label{general:complete:intersections:example}
The considerations of Example \ref{lang:conjecture:bounded:degree:example} will apply more generally for nonsingular complete intersection curves and also for higher dimensional nonsingular complete intersections.  However, for the case of non-singular complete intersection curves, in applying Theorem \ref{lang:conjecture:bounded:degree} along the lines just described (for suitable values of $n$ and $r$ that satisfy the inequality \eqref{canonical:bundle:formula:eqn}) we note that the conclusion will not improve on the known results about degrees of arithmetic irrationality from \cite{Smith:Vogt:2022}.  
\end{example}

\begin{example}\label{n:fold:Petri:example}
In dimension $1$, even after passing to a finite extension of the base number field $\KK$, the simplest instance of the relation  \eqref{lang:conjecture:bounded:degree:soltn:set:key:line:bundle:relation}, namely the question of writing the canonical bundle $\K_C$ of a curve $C$ as 
\begin{equation}\label{n:fold:petri:relation}
\K_C \simeq M^{\otimes n}
\end{equation}
for some very ample line bundle $M$ will not hold for Brill-Noether general curves.  Indeed, as follows from \cite{Grieve:Petri}, for such general curves $C$, it is not possible to write the canonical bundle $\K_C$ in the form \eqref{n:fold:petri:relation} for some very ample line bundle $M$ and $n \geq 3$.
\end{example}

\begin{example}\label{Extremal:curves:example}
Continuing to work in dimension $1$, as suggested by Examples \ref{lang:conjecture:bounded:degree:example} and \ref{n:fold:Petri:example}, it is interesting to study the hypothesis and conclusion of Theorem \ref{lang:conjecture:bounded:degree} for certain classes of curves that are not Brill-Noether general.  As a specific example, again passing to a finite extension of the base number field if necessary, consider the case of a non-plane extremal curve $C \subseteq \PP^s$ of degree $d = n(s-1)+2 \geq 2s +1$.  Here $n = \left\lfloor \frac{d-1}{s-1} \right\rfloor $.  Then $\K_C \simeq \Osh_C((n-1)H)$, for $H \subset C$ a hyperplane divisor.  Further, $C$ admits a complete $g^1_{n+1}$.  (We refer to \cite[Corollary III 2.6]{ACGH} for more details.)  As a consequence, we have an upper bound for the gonality of $C$, namely $\operatorname{gon}(C) \leq n + 1$.  On the other hand, we can apply Theorem \ref{lang:conjecture:bounded:degree} for degrees $r < \frac{n+2}{2}$.  Its conclusion, for the case that $\KK = \QQ$ give that, under the assumption of Conjecture \ref{Vojta:Main:Conj:Disc:Bounded:Degree}, that 
$$\operatorname{a.irr}_{\QQ}(C) \geq \frac{n}{2}+1 > \frac{\operatorname{gon}_{\QQ}(C)}{2} \text{.}$$   Such considerations should be compared with the general inequalities 
$$
\frac{\operatorname{gon}_{\KK}(C)}{2} \leq \operatorname{a.irr}_{\KK}(C) \leq \operatorname{gon}_{\KK}(C) \text{.}
$$
that are discussed and studied in \cite[p.~424]{Smith:Vogt:2022}.
\end{example}

Finally, we prove Theorem \ref{DM:Bombieri-Lang}.  Let us mention that Theorem \ref{DM:Bombieri-Lang} can be deduced from Theorem \ref{admissible:pair:DM:stack:conj} using Lemma \ref{height:function:lower:bound:discriminant:term}.  On the other hand, we deduce it here more directly as an application of Theorem \ref{lang:conjecture:bounded:degree}.

\begin{proof}[Proof of Theorem \ref{DM:Bombieri-Lang}]
The set
\begin{equation}\label{bounded:degree:set:cover:points:stack}
\left\{\text{$(\mathcal{D}_0,S)$ integral points $x \in \mathcal{X}(\overline{\KK}) : [\KK(x):\KK] \leq d$} \right\}
\end{equation}
is contained in the image, under $\pi$, of the set
\begin{equation}\label{bounded:degree:set:cover}
\left\{\text{$(D'_0,S)$-integral points $y \in Y(\overline{\KK}) : [\KK(y):\KK] \leq d  \cdot \operatorname{deg}(\pi)$} \right\} \text{.}
\end{equation}

Let $\operatorname{Ram}(\pi)$ be the ramification divisor.  Then $\K_Y + D_0' = \pi^*(\K_{\mathcal{X}} + \mathcal{D}_0) + \operatorname{Ram}(\pi)$ is a big line bundle on $Y$.   (Compare with \cite[Lemma 2.3]{Ascher:Javanpeykar:2017}.) 

The conclusion of Theorem \ref{DM:Bombieri-Lang} follows from the conclusion of Theorem \ref{lang:conjecture:bounded:degree} applied to the case that 
$r := (d \cdot \operatorname{deg}(\pi)) [\KK:\QQ]$.
\end{proof}

\providecommand{\bysame}{\leavevmode\hbox to3em{\hrulefill}\thinspace}
\providecommand{\MR}{\relax\ifhmode\unskip\space\fi MR }
\providecommand{\MRhref}[2]{%
  \href{http://www.ams.org/mathscinet-getitem?mr=#1}{#2}
}
\providecommand{\href}[2]{#2}

\end{document}